\documentclass[12pt]{article}
\usepackage[pdftex,pagebackref,letterpaper=true,colorlinks=true,
            pdfpagemode=none,urlcolor=blue,linkcolor=blue,
            citecolor=blue,pdfstartview=FitH]  {hyperref}

\usepackage{amsmath,amsfonts}
\usepackage{amsthm}
\usepackage{graphicx}
\usepackage{wrapfig}
\usepackage{color}

\theoremstyle{theorem}
\newtheorem{theorem}{Theorem}
\newtheorem{proposition}{Proposition}
\newtheorem{lemma}{Lemma}

\theoremstyle{definition}

\usepackage[table]{xcolor}
\definecolor{Gray}{gray}{0.85}

\usepackage{relsize}

\newcommand{\half}{\textstyle{\frac{1}{2}}}
\newcommand{\ord}{\mathop\mathrm{ord}}

\newcommand{\subgrp}{\text{\large$\trianglelefteq$}}

\newcommand{\NN}{\mathbb{N}}
\newcommand{\ZZ}{\mathbb{Z}}
\newcommand{\QQ}{\mathbb{Q}}
\newcommand{\DD}{\mathbb{D}}
\newcommand{\DK}{\mathbb{D}_K}

\newcommand{\Dbf}{\mathit{D}}
\newcommand{\Qbf}{\mathit{Q}}

\newcommand{\NE}{\mathbb{N}_\mathrm{E}}
\newcommand{\NO}{\mathbb{N}_\mathrm{O}}
\newcommand{\ZE}{\mathbb{Z}_\mathrm{E}}
\newcommand{\ZO}{\mathbb{Z}_\mathrm{O}}

\newcommand{\QE}{\mathbb{Q}_\mathrm{E}}
\newcommand{\QK}{\mathbb{Q}_\mathrm{K}}
\newcommand{\QO}{\mathbb{Q}_\mathrm{O}}
\newcommand{\QN}{\mathbb{Q}_\mathrm{N}}
\newcommand{\QP}{\mathbb{Q}_\mathrm{P}}

\begin{document}

\title{Parity and Partition of  the Rational Numbers} 

\author{Peter Lynch and Michael Mackey \\
            School of Mathematics \&\ Statistics, University College Dublin%
           \footnote{Peter.Lynch@ucd.ie; mackey@maths.ucd.ie.} }

\maketitle



\begin{abstract}
We define an extension of parity from the integers to the rational
numbers. Three parity classes are found --- even, odd and `none'.
Using the 2-adic valuation, we partition the rationals into subgroups
with a rich algebraic structure.

The natural density provides a means of distinguishing the sizes
of countably infinite sets.  The Calkin-Wilf tree has a remarkably
simple parity pattern, with the sequence `odd/none/even' repeating
indefinitely. This pattern means that the three parity classes have
equal natural density in the rationals.
A similar result holds for the Stern-Brocot tree.
\end{abstract}


\noindent
The natural numbers $\NN$ split nicely into two subsets, the odd and even numbers
\begin{equation}
\NO = \{ 1, 3, 5, 7, \dots \} \,, \qquad \NE = \{ 2, 4, 6, 8, \dots \} \,.
\nonumber
\end{equation}
Stopping at some number $2N$, the odd and even numbers are equinumerous.
Stopping at $2N+1$, the odds are slightly ahead,
but as $N$ gets larger, the ratio of odd to even numbers tends to $1$.
So, we can say informally that there are the same number of odd and even integers.
This will be made precise below by defining densities for the sets
$\NO$ and $\NE$.
Similar arguments apply to the integers $\ZZ$, which split into two subsets
\begin{eqnarray*}
\ZO &=& \{\dots\  -3, -1, +1, +3, +5, \dots \} \\
\ZE &=& \{\dots\  -4, -2,\ \ 0, +2, +4,  \dots \} \,.
\end{eqnarray*}
The integers form an abelian group $(\ZZ,+)$ under addition. %
The even numbers form an additive subgroup of $(\ZZ,+)$,
with index $[\ZZ : \ZE] = 2$ and two cosets $\ZE$ and $\ZE+1 = \ZO$.
This definition provides a bijection between the two cosets, which have the same cardinality.%

\section{Parity.}
\label{sec:parity}

The distinction between odd and even numbers is called \emph{parity}.
The even/odd concept is defined only for the integers. The distinction 
does not apply to fractions or irrational
numbers, but one may wonder if there is a natural way to extend the
concept of  parity to larger sets of numbers.

What characteristics might one require of such an extension?  The 
definition would have to agree with the traditional definition for the
integers, so $5$ would continue to be odd and $10$ even. In addition,
the usual `rules of parity' might be required:
\begin{enumerate}
\item The sum of two even numbers is even; the product is even.
\item The sum of two odd numbers is even; the product is odd.
\item The sum of an even and an odd number is odd; the product is even.
\item An odd number plus $1$ is even; an even number plus $1$ is odd.
\end{enumerate}
Table \ref{tab:AddMultZ} shows the effects of addition and multiplication
on the ring of integers.
%


\begin{table}
\begin{center}
\caption{Addition table (left) and multiplication table (right) for $\ZZ$.}
\label{tab:AddMultZ}

\bigskip
\begin{minipage}[c]{0.45\textwidth}
\begin{center}
\begin{tabular}{||c||c|c||}   
   
\hline \hline 
        $\boldsymbol{+}$    &   $\mathbf{even}$  &  $\mathbf{odd}$  \\
\hline \hline
   $\mathbf{even}$   &   $even$  &  $odd$  \\ 
   $\mathbf{odd}$     &   $odd$   &  $even$  \\ 
\hline  \hline   

\end{tabular}   
\end{center}
\end{minipage}
\begin{minipage}[c]{0.45\textwidth}
\begin{center}
\begin{tabular}{||c||c|c||}   
 
\hline \hline 
   $\boldsymbol{\times}$       &   $\mathbf{even}$  &    $\mathbf{odd}$  \\
\hline \hline
       $\mathbf{even}$     &   $even$  &  $even$  \\ 
       $\mathbf{odd}$       &  $even$    &  $odd$  \\ 
  \hline  \hline   
     
\end{tabular}   
\end{center}
\end{minipage}
\end{center}
\end{table}   

If the concept of parity is extended to larger sets of numbers,
some of the properties indicated above may have to be sacrificed.
For rational numbers, we might define a number $q=m/n$ to be even
if the numerator $m$ is even and odd if $m$ is odd. But 
then $\frac{1}{4}+\frac{1}{4} = \frac{1}{2}$,
meaning that two odd rationals would add to yield another odd one.

We will distinguish between `odd' and `uneven' rationals:
assuming $m$ and $n$ to be relatively prime integers, $(m,n)=1$,
we will adopt the following definition: 
\begin{equation}
\mbox{For any rational number\ } q=\frac{m}{n} \,,
\quad
\begin{cases}
q \ \mbox{is \emph{even} if $m$ is even}, \\
q \ \mbox{is \emph{uneven} if $m$ is odd} \,.
\end{cases}
\label{eq:def1}
\end{equation}

\section*{A three-way split.}
\label{sec:3way}

There is a simple way of separating the rational numbers into three subsets:
\begin{equation}
\mbox{For any rational\ } q=\frac{m}{n} \,,
\quad
\begin{cases}
q \ \mbox{has parity \emph{even} if $m$ is even and $n$ is odd,} \\
q \ \mbox{has parity \emph{odd} if $m$ is odd and $n$ is odd,} \\
q \ \mbox{has parity \emph{none} if $m$ is odd and $n$ is even.}
\end{cases}
\label{eq:def2}
\end{equation}
The term \emph{none} is an initialism for `neither odd nor even'.
Corresponding to this three-way partition, we define three subsets of the rationals:
\begin{eqnarray*}
\mbox{Even:\ \ }
\QE &=& \{ q\in\QQ :
q = \textstyle{\frac{2m}{2n+1}\ \mbox{for some}\ m,n\in\ZZ} \} \\
\mbox{Odd:\ \ }
\QO &=& \{ q\in\QQ :
q = \textstyle{\frac{2m+1}{2n+1}\ \mbox{for some}\ m,n\in\ZZ} \} \\
\mbox{None:\ \ }
\QN &=& \{ q\in\QQ :
q = \textstyle{\frac{2m+1}{2n}\ \mbox{for some}\ m,n\in\ZZ} \} \,.
\end{eqnarray*}
These three sets are mutually disjoint and
$\QQ = \QE \uplus \QO \uplus \QN$.
It is immediately obvious that $\ZE \subset \QE$ and
$\ZO \subset \QO$, confirming that the definition of parity
for the rationals is an extension of the usual meaning for the integers.
We see that the even and odd rationals
respect the four `rules of parity' listed above.

\subsection{``Twice as many uneven as even fractions''.}
\label{sec:twiceasmany}

The rational numbers are countable: they can be put into one-to-one
correspondence with the natural numbers.  We can list all rationals
in $(0,1)$ in a sequence where, for each $n$ in turn, all (new) numbers
$m/n$ with $m < n$ are listed in order. For $n_\mathrm{max}=8$ we have
\begin{eqnarray}
\bigl\{
\textstyle{
\frac{1}{2},\frac{1}{3},\frac{2}{3},\frac{1}{4},\frac{3}{4},
\frac{1}{5},\frac{2}{5},\frac{3}{5},\frac{4}{5},\frac{1}{6},\frac{5}{6},
\frac{1}{7},\frac{2}{7},\frac{3}{7},\frac{4}{7},\frac{5}{7},\frac{6}{7},
\frac{1}{8},\frac{3}{8},\frac{5}{8},\frac{7}{8}}
\bigr\} \,.
\label{eq:farish}
\end{eqnarray}
Rearrangement in increasing order of magnitude gives the Farey sequence $F_8$.
%

A \textsc{Mathematica} program was written to count the proportion
of rationals in each parity class in the interval $(0,1)$,
with denominators less than or equal to $n$, for a range of cut-off values
$n \le n_\mathrm{max}$.  The ratios are plotted in Fig.~\ref{fig:pratio20}.
As $n$ increases, the ratios of numbers with parity even, odd and none all
tend to the limit $\frac{1}{3}$. 
Colloquially, there are an equal number of rationals with parity even, odd and none,
and ``twice as many uneven as even rationals''.
%
%
\begin{figure}[h]
\begin{center}
\includegraphics[width=0.95\textwidth]{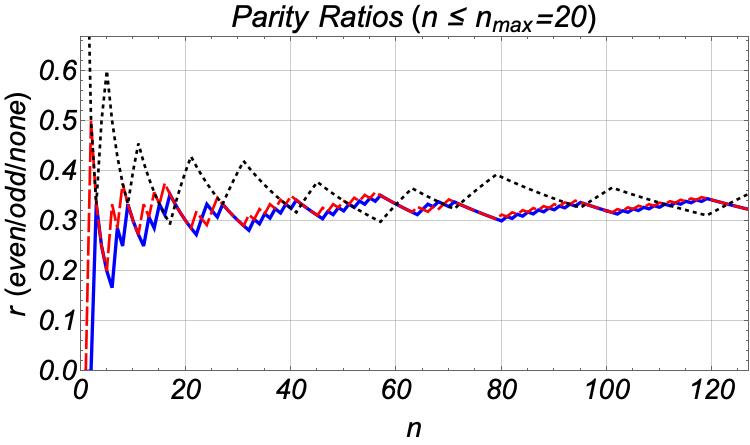} 
\caption{Parity ratio $r$ for rationals $m/n$ of parity even (solid line, blue online),
odd (dashed line, red online) and none (dotted line, black online) for $n\le n_{\mathrm{max}}=20$.}
\label{fig:pratio20}
\end{center}
\end{figure}

\section{The Density of subsets of $\mathbb{N}$.}
\label{sec:density}

In pure set-theoretic terms, the set of even positive numbers is ``the same size''
as the set of all natural numbers; both are infinite countable sets.  However,
cardinality is a blunt instrument: with the usual ordering, every second natural number
is even and, intuitively, we feel that there are half as many even numbers as natural numbers.  
The concept of \emph{density} provides a means of expressing the relative sizes
of sets that is more discriminating than cardinality.

Density --- also called natural or asymptotic density --- is defined for many
interesting subsets of $\mathbb{N}$, although not for all subsets. Assume a 
subset $A$ of $\mathbb{N}$ is enumerated as $\{a_1, a_2, \dots \}$.
We define the density of $A$ in $\mathbb{N}$ as the limit, if it exists,
\begin{equation}
\rho_\mathbb{N}(A) = \lim_{n\to\infty}\ \frac{|\{a_k : a_k\le n\}|} {n} \,.
\label{eq:defdenN}
\end{equation}
Thus, if the fraction of elements of $A$ among the first $n$ natural numbers
converges to a limit $\rho_\mathbb{N}(A)$ as $n$ tends to infinity, then $A$
has density $\rho_\mathbb{N}(A)$ \cite{Te}.
More generally, if $A = \{a_1, a_2, a_3, \dots \}$
is a subset of a countable set $X$ enumerated as
$\{x_1,x_2,x_3, \dots \}$, we define
the density of $A$ in $X$ --- if it exists --- as
\begin{equation}
\rho_{X}(A) = \lim_{n\to\infty}\ \frac{|A\cap \{x_1, x_2, \dots ,  x_n\}|}{n} \,.
\label{eq:defdenZ}
\end{equation}
For $X = \mathbb{N}$, we usually write $\rho_\mathbb{N}(A)$ as $\rho(A)$.
For $A = \NE$ or $A = \NO$, we have $\rho(A) = \half$, as might be expected. 
This is consistent with our intuitive notion that $50\%$ of the natural numbers
are even and $50\%$ are odd.

Let us now rearrange the natural numbers into a set $F$ such that there are
\emph{twice as many even as odd numbers in $F$}.
We reorder $\mathbb{N}$ so that each odd number is followed by two even ones:
$$
F = \{ 1, 2, 4, 3, 6, 8, 5, 10, 12,\ \dots\ , 2n-1, 4n-2, 4n, \dots \} \,.
$$
It is easy to see that $\rho_F(\NE) = \frac{2}{3}$ and $\rho_F(\NO) = \frac{1}{3}$.
Proceeding further, we can construct a set $H$ in which the $n$-th odd number is followed by
$n$ even numbers.  We find that $\rho_H(\NE)=1$,
so that ``almost all the elements of $H$ are even''.
 
These examples make it clear that density depends strongly on the ordering of the reference set.
Our intuition is guided by the usual (natural) ordering of the natural numbers and the alternation
between odd and even numbers leads us to the conclusion that, somehow, they are equal in number,
each comprising ``half'' of the set of natural numbers. Density relative to $\mathbb{N}$ is consistent 
with this intuition.

With the ordering $\{0, +1, -1, +2, -2, \ldots \}$ of the integers,
the densities defined by (\ref{eq:defdenZ}) are
$$
\rho_\ZZ(\ZE) = \half \,, \qquad \rho_\ZZ(\ZO) = \half \,.
$$
We will prove that, for the rational numbers with the Calkin-Wilf and
Stern-Brocot orderings defined below,
\begin{equation}
\rho_\QQ(\QE) = \rho_\QQ(\QO) = \rho_\QQ(\QN) = \tfrac{1}{3} \,.
\label{eq:DenOneThird}
\end{equation}


\section{Partitioning the rationals.}
\label{sec:paritioning}

\begin{table}[h]
\begin{center}
\caption{Addition table (left) and multiplication table (right) for $\QQ$.
The entry `\emph{any}' indicates that the result may be in any of the three parity classes.}
\label{tab:qaddmul}
   
\bigskip
\begin{minipage}[c]{0.45\textwidth}
\begin{center}
\begin{tabular}{||c||c|c||c||}   
   \hline \hline 
   $\boldsymbol{+}$      &   $\mathbf{even}$  &  $\mathbf{odd}$   &  \cellcolor{Gray}  $\mathbf{none}$  \\
   \hline \hline
   $\mathbf{even}$   &   $even$  &  $odd$   &  \cellcolor{Gray}  $none$  \\ 
   $\mathbf{odd}$    &   $odd$   &  $even$  &   \cellcolor{Gray}  $none$  \\ 
   \hline\hline
  \cellcolor{Gray}   $\mathbf{none}$   &  \cellcolor{Gray}   $none$  &  \cellcolor{Gray}  $none$  &  \cellcolor{Gray}  $\underline{any}$  \\
   \hline \hline   
\end{tabular}   
\end{center}
\end{minipage}
\hspace{3mm}
\begin{minipage}[c]{0.45\textwidth}
\begin{center}
\begin{tabular}{||c||c|c||c||}   
   \hline \hline 
   $\boldsymbol{\times}$  &  $\mathbf{even}$   &   $\mathbf{odd}$   &  \cellcolor{Gray}  $\mathbf{none}$  \\
   \hline \hline
   $\mathbf{even}$    &  $even$   &   $even$  & \cellcolor{Gray}   $\underline{any}$   \\
   $\mathbf{odd}$     &  $even$   &   $odd$   &  \cellcolor{Gray}  $none$  \\ 
   \hline\hline
  \cellcolor{Gray}   $\mathbf{none}$    &   \cellcolor{Gray}  $\underline{any}$   &  \cellcolor{Gray}   $none$  &  \cellcolor{Gray}  $none$  \\ 
   \hline \hline   
\end{tabular}   
\end{center}
\end{minipage}

\end{center}
\end{table}   

In Table~\ref{tab:qaddmul} we show the results of adding and multiplying
numbers from the three parity classes.  The most important thing
to notice is that, if we confine attention to only the even and odd
rationals, the tables are identical to the addition and multiplication
tables for $\ZZ$ (Table~\ref{tab:AddMultZ}).  The entry
\emph{``any''} in the tables indicates a number that is a ratio of two
even numbers and that may, after reduction, be in any of the
three parity classes.
Examining the left panel of Table~\ref{tab:qaddmul}, we see that $(\QE,+)$ 
is an additive (normal) subgroup of $(\QQ,+)$. 
In Table~\ref{tab:qaddmul} (right panel) we show the results of
multiplying numbers from the three parity classes.
Restricting attention to the even and odd rationals only --- omitting those with
no parity --- we define
\begin{equation} 
\QP := \QE \uplus \QO  \,.
\end{equation} 
This is the set of all rationals whose denominators are odd numbers in $\ZZ$. 
It is closed under addition and multiplication and forms a
commutative subring of the field $\QQ$.
Moreover, since there are no divisors of zero, $\QP$
is an integral domain \cite{DuFo}.   
Although $\QP$ is not an ideal of $\QQ$ (fields do not have proper ideals),
it is a (normal) subgroup of $(\QQ,+)$. So, we may enquire about its index 
$[\QQ : \QP]$ and its quotient group $\QQ / \QP$.%



Somewhat out of context, we mention the easily-proved observation that
all three parity classes, $\QE$, $\QO$ and $\QN$, 
are (topologically) dense in the rationals.

\subsection*{2-Adic valuation and the ``degree of evenness''.}

\begin{center}
\emph{All multiples of 2 are even, but}  
\emph{some are more even than others.}
\end{center}

The $p$-adic valuation --- or $p$-adic order \cite{Ka} --- of an integer $n$ is the function
$$
\nu_p(n) = \begin{cases}
\max\{k\in\NN : p^k \mid n \} & \mbox{for\ } n \neq 0 \\
\infty                               & \mbox{for\ } n   =  0   
\end{cases}
$$
This is extended to the rational numbers $m/n$:
$$
\nu_p\left(\frac{m}{n}\right) = \nu_p(m) - \nu_p(n) \,. 
$$
It is easily proved that, for any rationals $q_1$ and $q_2$, 
\begin{equation}
\nu_p(q_1+q_2) \ge \min\{\nu_p(q_1),\nu_p(q_2)\} \,,
\label{eq:nuineq}
\end{equation}
with equality holding if $\nu_p(q_1) \neq \nu_p(q_2)$.

We shall be concerned exclusively with the case $p = 2$. 
We note that $\QP = \{q\in\QQ : \nu_2(q) \ge 0 \}$
and $\QE = \{q\in\QQ : \nu_2(q)  >  0 \}$.
The ``degree of evenness'' of a number can be expressed in terms of the 2-adic valuation. 
For an integer $n$, the 2-adic valuation is the largest natural number $k$ such that
$2^k$ divides $n$. It is normally written $\nu_2(n)$ or $\ord_2(n)$.
For even integers, $\nu_2(n)>0$; for odd integers, $\nu_2(n)=0$.
By convention, $\nu_2(0)=\infty$ (since zero is divisible by every power of $2$).

If we write a rational number $q$ in the form $2^k(2m+1)/(2n+1)$ with $k\in\ZZ$,
then $\nu_2(q) = k$.  Odd rationals have order $0$ and rationals with no parity
have negative 2-adic order. In particular, half integers have 2-adic order equal to $-1$.
In summary,
\begin{eqnarray*}
\mbox{For $q$ rational with parity \emph{even},\ } & \nu_2(q) > 0 \,, \\ 
\mbox{For $q$ rational with parity \emph{odd},\  } & \nu_2(q) = 0 \,, \\ 
\mbox{For $q$ rational with parity \emph{none},\ } & \nu_2(q) < 0 \,.
\end{eqnarray*}
 
The 2-adic order clearly identifies the parity classes of the rationals, and it
provides a means of partitioning them into finer-grain parity classes. 
The resulting partition reveals a wealth of algebraic structure.
For all $k\in\ZZ$, we define the set of all rational numbers with 2-adic valuation $k$:
$$
\Qbf_k = \{q\in\QQ : \nu_2(q) = k \} \qquad\mbox{and}\qquad
\Qbf_\infty = \{\ 0\ \} \,.
$$
The union of all the $\Qbf$-sets comprises the entire set of rationals
$$
\QQ = \{\ 0\ \} \uplus \biguplus_{k=-\infty}^\infty \Qbf_k \,.
$$
We illustrate the subsets $\Qbf_k$ in Fig.~\ref{fig:QP-Plane-Lin}.
The vertical axis is the 2-adic valuation $\nu_2$. Each subset $\Qbf_k$
is represented by a horizontal dotted line.
We remark that all odd rationals are in $\Qbf_0$ and 
all even rationals are in $\bigcup_{k>0}\Qbf_k$.
For all $K\in\ZZ$, we define
\begin{equation}
\QK \equiv \{\ 0\ \} \uplus \biguplus_{k\ge K} \Qbf_k
\label{eq:QK}
\end{equation}
and observe that $\QK$ is a subgroup of $\QQ$.  We write this as $\QK\ \subgrp\ \ \QQ$.
Note, in particular, that $\QQ_0 = \QP$ and $\QQ_1 = \QE$.
%
%
There is an infinite chain of subgroups, starting with $\QQ_\infty = \{\ 0\ \}$
and extending through all the $\QQ_K$ groups to the full group of rationals:
$$
\QQ_\infty = \{\ 0\ \} \ \subgrp\ \ \cdots\ \subgrp\ \ \QQ_{2}\ \subgrp\ \ \QQ_{1}\
\subgrp\ \ \QQ_{0}\ \subgrp\ \ \QQ_{-1}\ \subgrp\ \ \QQ_{-2}\ \subgrp\ \ \cdots\ \subgrp\ \ \QQ \,.
$$

\subsection*{Dyadic rational numbers.}

A dyadic rational is a number that can be expressed as a fraction whose
denominator is a power of two.  The usual definition of the dyadic
rational numbers \cite[pg.~122]{Ba} is  
$$
\DD = \left\{ \frac{z}{2^m} : z\in\ZZ, m\in\ZZ \right\} \,.
$$
Note that the integers are included in the set of dyadic rationals. 
A convenient alternative definition is
$$
\DD = \{ 2^k(2\ell-1) : k\in\ZZ, \ell\in\ZZ \} \uplus \{\ 0\ \} \,,
$$
since all the numbers of the form $2^k(2\ell-1)$ are in $\Qbf_k$.
Moreover, the expression of each number in this form is unique.
We also define the sets
$$
\Dbf_k = \{ 2^k(2\ell-1) : \ell\in\ZZ \} \qquad\mbox{and}\qquad
\Dbf_\infty = \{\ 0\ \} \,,
$$
and note that $\DD = \biguplus_{k} \Dbf_k \,\uplus \,\Dbf_\infty$.
We see that $\Dbf_0 = \ZO$ and $\biguplus_{k>0}\Dbf_k = \ZE$.

The dyadics correspond to all real numbers with finite binary expansions,
and also to the set of surreal numbers born on finite days \cite{Co}.
The dyadic rational numbers form a ring between the ring of
integers and the field of rational numbers:
$$
\ZZ \ \subgrp\ \DD \ \subgrp\ \QQ \,.
$$

The sets $\Dbf_k$ are indicated in Fig.~\ref{fig:QP-Plane-Lin} by the marked points in $\Qbf_k$.
The vertical axis is the 2-adic valuation $\nu_2$.
For each $k$,  $\Dbf_k\subset\Qbf_k$.
The (black) dots at level $k=0$ are the odd integers.
The (blue) dots at positive $k$-levels are the even integers.
The (red) dots at each negative level $k<0$ are the dyadic fractions,
with odd numerator and denominator $2^{k}$.
Zero sits, like an angel, on top of the tree.

By analogy with the definition (\ref{eq:QK}) of the $\QQ_K$-sets,
we construct a countable infinity of subgroups of $\DD$:
$$
\DD_K := \{\ 0\ \} \uplus \biguplus_{k\ge K} \Dbf_k \,.
$$
Particular cases of the $\DD$-sets include
$$
\DD_{-\infty} = \DD \,, \qquad \DD_{-1} = \half\ZZ \,, \qquad 
\DD_0 = \ZZ \,, \qquad \DD_1 = \ZE \,, \qquad \DD_\infty = \{\ 0\ \} \,.
$$
There is an infinite chain of subgroups starting with $\DD_\infty$
and extending through all the $\DK$ groups to the full group of dyadic rationals:
$$
\DD_\infty = \{\ 0\ \} \ \subgrp\ \ \cdots\ \subgrp\ \ \DD_{2}\ \subgrp\ \ \DD_{1}\
\subgrp\ \ \DD_{0}\ \subgrp\ \ \DD_{-1}\ \subgrp\ \ \DD_{-2}\ 
\subgrp\ \ \cdots\ \subgrp\ \ \DD \,.
$$

\begin{figure}[t]
\begin{center}
\includegraphics[width=0.75\textwidth]{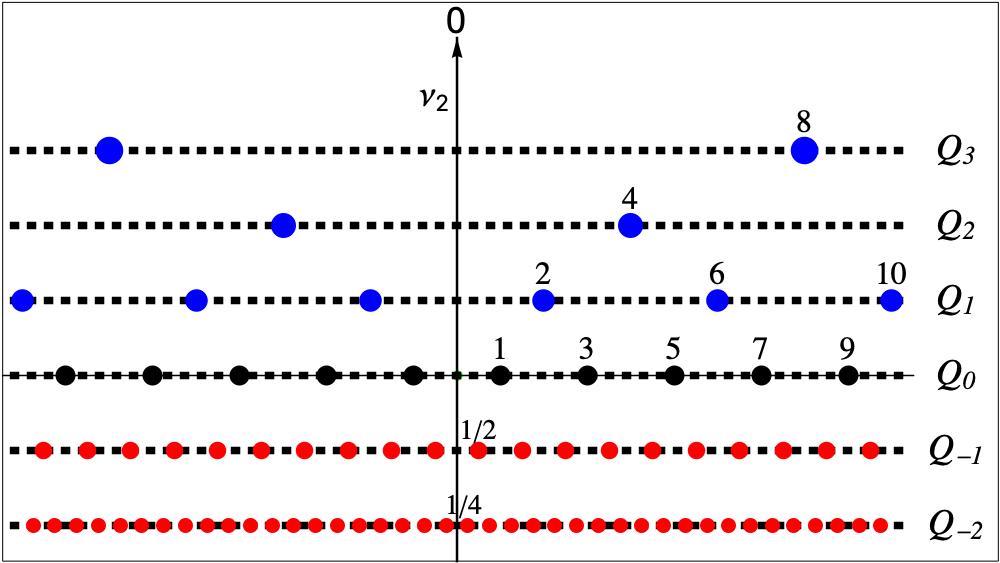}
\caption{Partition of the rational numbers.
The vertical axis is the 2-adic valuation $\nu_2$.
Each (dense) subset $\Qbf_k$ is represented by a horizontal dotted line.
The sets $\Dbf_k$ are indicated by the marked points in $\Qbf_k$.
The totality of these comprises the dyadic rationals $\DD$.}
\label{fig:QP-Plane-Lin}
\end{center}
\end{figure}

Readers familiar with the theory of $p$-adic numbers may wish to show that
$\QP$ is the ring of rational-valued 2-adic integers, $\QQ\cap\ZZ_2$,
and the dyadic rational numbers may be expressed as
$$
\mathbb{D} = \QQ \cap \bigcap_{p\ \mathrm{odd}} \ZZ_p \,.
$$

\section{Cosets of $\QP$ in $\QQ$.}
\label{sec:cosets}

In the following section we show that any two rationals $q_1$ and $q_2$ with distinct,
negative 2-adic orders are representatives of distinct cosets: $q_1 + \QP \neq q_2 + \QP$.
Thus, if $q_1 + \QP = q_2 + \QP$ then $\nu_2(q_1)=\nu_2(q_2)$. Consequently, there is
at least one coset for each $k<0$ and therefore an infinite number of cosets. However,
it is clear that 
$\nu_2(q_1) = \nu_2(q_2)$ does not imply equality of cosets; consider, for example,
$q_1 =\frac{1}{4}$ and $q_2 = \frac{3}{4}$, since $(q_2-q_1)=\half\not\in \QP$.

We now investigate the cosets $q + \QP$ in $\QQ / \QP$.
First, we note that if $q \in \Qbf_{-k}$ then $q + \QP \subset \Qbf_{-k}$,
but there is no guarantee that $q + \QP$ is equal to $\Qbf_{-k}$.
Suppose $q_1$ and $q_2$ are in $\Qbf_{-k}$ for some $k>0$. If they represent the same coset
then $q_1 - q_2 \in \QP$. However, it is easily seen that $\nu_2(q_1-q_2)$ may assume any value
greater than $-k$:
$$
q_1 - q_2 = 2^{\ell-k} \frac{o_1}{o_2} \quad\mbox{for}\quad \ell > 0 \,.
$$
Thus, $q_1 - q_2$ may be in any of the following sets:
$$
\Qbf_{-k+1} \,, \Qbf_{-k+2} \,, \Qbf_{-k+3}\,, \dots \,, \Qbf_{-1} \,, \QP \,.
$$
Clearly, $q_1 - q_2 \in \QP$ if and only if $\ell \ge k$, whence
$$
q_1 + \QP = q_2 + \QP \qquad \mbox{\textrm{iff}} \qquad \ell \ge k \,.
$$

For each $k>0$, we define a set of values 
\begin{equation}
q_{k}^{\ell} = 2^{-k}(2\ell-1) \in \Qbf_{-k}
\quad\mbox{for}\quad  \ell=1,2,3, \dots , 2^{k-1} \,.
\label{eq:qkl}
\end{equation}
We note that these are the first $2^{k-1}$ positive values in $\Dbf_{-k}$.
We show in the following section that these are representatives of $2^{k-1}$
cosets, which are all distinct and which provide a disjoint partition of $\Qbf_{-k}$.
This analysis provides explicit expressions for each of the infinite set
of cosets of $\QP$ in $\QQ$:
\begin{equation}
 q_{k}^{\ell}+\QP \,, \qquad \ell=1, 2, 3, \dots 2^{k-1} \,, \quad k=1,2, \dots  \,. 
\end{equation}

\subsection*{Scale invariance of the structure.}

\begin{figure}[h]
\begin{center}
\includegraphics[width=0.75\textwidth]{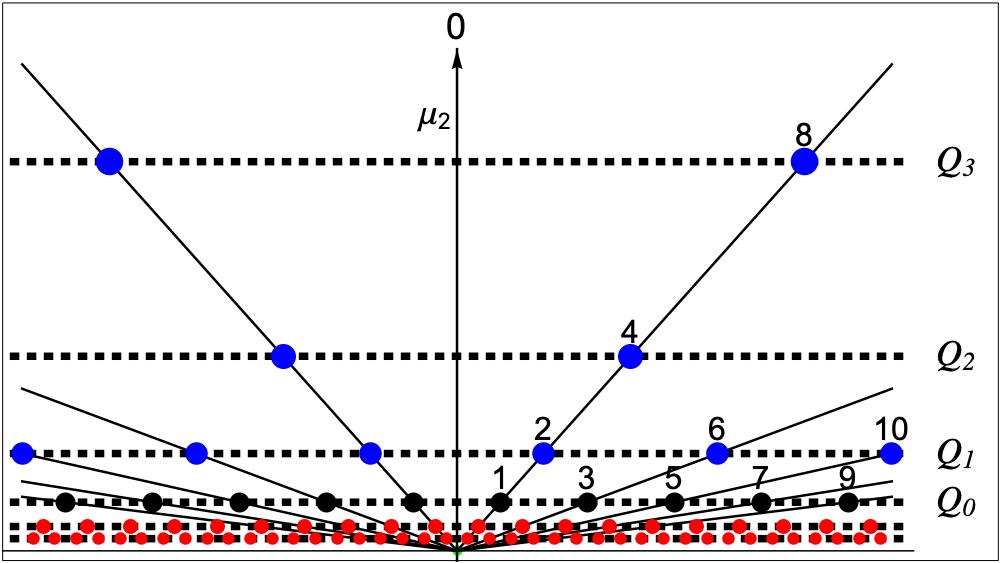}
\caption{Partition of the rational numbers. 
The vertical axis is $\mu_2 = 2^{\nu_2}$.
Each subset $\Qbf_k$ is represented by a horizontal dotted line.
The sets $\Dbf_k$ are indicated by the marked points in $\Qbf_k$.
The dyadic numbers in $\DD$ all appear on a pencil of lines emanating
from the origin.}
\label{fig:QP-Plane-Log}
\end{center}
\end{figure}

We notice that the diagram in Fig.~\ref{fig:QP-Plane-Lin} has a scaling invariance:
if the horizontal axis is stretched by a factor of 2 and the diagram translated
one unit in the vertical, the dyadic rationals occupy the same set of points.
We have chosen to analyse the quotient group $\QQ/\QP$. However, a similar analysis
could be done for any subgroup $\QK$, with directly analogous results.

In Fig.~\ref{fig:QP-Plane-Log}, we re-plot Fig.~\ref{fig:QP-Plane-Lin} 
with vertical axis $\mu_2 = 2^{\nu_2}$.  The dyadic numbers in $\DD$ now appear
on a pencil of lines with slopes $1/(2\ell-1)$, emanating from the origin.



\subsection*{Density of $\Qbf_{k}$: a heuristic discussion.}

The set $\Qbf_{-1} = \half + \QP$ is a coset of $\QP$. It can be
visualized as a copy of $\QP$ shifted by a distance $\half$.
We argue heuristically that $\Qbf_{-1}$ is ``as dense as $\QP$''.

More generally, for any $k$, there is a natural correspondence between
elements of $\Qbf_{k}$ and elements of $\Qbf_{k-1}$:
$$
\frac{1}{2^{k}}  \left(\frac{2m+1}{2n+1}\right) \in \Qbf_{k} \longleftrightarrow 
\frac{1}{2^{k-1}}\left(\frac{2m+1}{2n+1}\right) \in \Qbf_{k-1}  \,.
$$
Thus, $\Qbf_{k-1}$ may be visualized as a compressed version of $\Qbf_{k}$.
Since $\Qbf_{k-1}$ is ``twice as dense as $\Qbf_{k}$'', we may argue that we
should have twice as many cosets in $\Qbf_{k-1}$ as there are in $\Qbf_{k}$.
This is consistent with what is proved rigorously below. 

\section{Formal proof of the coset structure for $\QP$.}
\label{sec:proofs}

In this section, we give rigorous proofs of some of the results
considered heuristically in the discussion above.
Lemmas 1 and 2 give conditions for cosets to be equal. Proposition 1
gives explicit representatives $q_k^\ell$ for each of the distinct
cosets.  In the following, we abbreviate the $2$-adic valuation $\nu_2$ to $\nu$.

\begin{lemma}
Suppose $q_1, q_2 \in \QN$ and $q_1+\QP = q_2+\QP$.  Then $\nu(q_1)=\nu(q_2)$.
\end{lemma}

\begin{proof}
For the cosets to be equal, we must have $q_1-q_2 \in\QP$. Suppose that
$\nu(q_1)=-k$ and $\nu(q_2)=-\ell$ with $k\ne\ell\in\NN$. Without loss of generality,
we may assume that $\ell = k+d$ with $d>0$. 
Then, using (\ref{eq:nuineq}), 
\begin{equation}
\nu(q_1-q_2) = \min(\nu(q_1), \nu(-q_2)) = \min(\nu(q_1), \nu(q_2)) = \min(-\ell, -k) = -\ell \,.
\nonumber
\end{equation}
Thus, $q_1-q_2 \in \Qbf_{-\ell}$,
so that $q_1+\QP \ne q_2+\QP$.
Consequently, a necessary condition for equality of the cosets is that $q_1$ and $q_2$
have the same 2-adic valuation, $\nu(q_1)=\nu(q_2)$. 
\end{proof}

The next lemma strengthens this to a necessary and sufficient condition.

\begin{lemma}
Let $q_1$ and $q_2$ be in $\QN$.  Then $q_1+\QP=q_2+\QP$ if, and only
if, $\nu(q_1)=\nu(q_2)=-k$ for some $k\in \NN$ and
$\nu(a_1b_2-a_2b_1))\ge k$ where $q_1={a_1}/{(2^kb_1)}$ and
$q_2={a_2}/{(2^kb_2)}$.
\end{lemma}

\begin{proof}
Let $q_1$ and $q_2$ be in $\QN$.  Suppose that $q_1+\QP=q_2+\QP$.
Lemma 1 tells us that $\nu(q_1)=\nu(q_2)=-k$ for some $k\in \NN$.   We
may thus write 
\begin{equation}
q_1= \frac{a_1}{2^kb_1} \quad \mathrm{and} \quad
q_2=\frac{a_2}{2^kb_2}
\label{eq:0}
\end{equation}
where
$a_i, b_i$ are odd, and so
\begin{equation}
  q_1-q_2 = \frac{a_1 b_2 - a_2 b_1}{2^k b_1 b_2}.
  \label{eq:00}
\end{equation}
By hypothesis, this is an element of $\QP$ which implies
  $2^k |(a_1b_2-a_2b_1)$, as asserted.

  For the converse, if $\nu(q_1)=\nu(q_2)=-k$, then \eqref{eq:0}, and
  hence \eqref{eq:00}, holds.   The condition $\nu(a_1b_2-a_2b_1)\ge
  k$ then implies $\nu(q_1-q_2)\ge 0 $ and so $q_1-q_2\in \QP$.
\end{proof}


\begin{proposition}
  For each $k\in\NN$, let $q_k^\ell = ({2\ell-1})/{2^k} \in \Qbf_{-k}$ for
  $\ell=1,\ldots, 2^{k-1}$. These numbers generate $2^{k-1}$ distinct $\QP$-cosets,
  which comprise all the cosets of $\QP$ by elements of $\Qbf_{-k}$. 
\end{proposition}

\begin{proof}
  It is clear that $q_k^\ell = {(2\ell-1)}/{2^k}$ lies in
  $\Qbf_{-k}$.   If two of these numbers, $q_k^\ell$ and $q_k^{\ell'}$ say, generate
  the same coset, then $q_k^\ell - q_k^{\ell'} \in \QP$ and so, by Lemma 2,
  $2^k $ divides $2{\ell'}-1 - (2\ell-1)=2({\ell'}-\ell)$ and hence $2^{k-1}$ divides
  $\ell-\ell'$.  As $1\le \ell,\ell'\le 2^{k-1}$, the only way this can occur is if
  $\ell=\ell'$. This shows that all $2^{k-1}$ cosets generated by $q_k^\ell$ are distinct.

  Next, we show that these exhaust all possible cosets by elements of
  $\Qbf_{-k}$.  For this, given $q\in \Qbf_{-k}$ we need to show that
  $q+\QP= q_{k}^{\ell}+\QP$ for some $\ell\in\{1,2,\ldots 2^{k-1}\}$.  As $q = n/(2^k m)$
  for $m$ and $n$ odd, by Lemma 2, this amounts to showing
  $\nu(1.n - (2\ell-1)m) \ge k$, or that $2^k$ divides $n+m-2\ell m$.
  We may let $n+m=2s$ ($s\in\ZZ$) and examine instead whether $2^{k-1}$
  divides $s-\ell m$ for some $\ell\in\{1,\ldots,2^{k-1}\}$.

  Notice that
  \begin{align*}
  s-\ell m \equiv s-\ell' m ( \!\!\!\!\!  \mod 2^{k-1} )   &\iff   2^{k-1} \mbox{\ divides\ } m(\ell-\ell')   \\
                              & \overset{\mbox{$m$\ odd} } \iff    2^{k-1} \mbox{\ divides\ } (\ell-\ell')    \iff \ell = \ell'
  \end{align*}
  as again $1\le \ell,\ell'\le 2^{k-1}$.
  In particular, by the pigeonhole
  principle, all $2^{k-1}$ possible $2^{k-1}$-remainders, including 0, are
  contributed by $s-\ell m$, $\ell =1,\ldots,2^{k-1}$.   So $2^{k-1}$ divides
  $s-\ell m$ for some $\ell$ and, for this $\ell$ we have $q\in q_k^\ell+\QP$, as stated.
\end{proof}



\section{Densities of the parity classes.}
\label{sec:CW-Tree}

\begin{figure}[t]
\begin{center}
\includegraphics[width=0.45\textwidth]{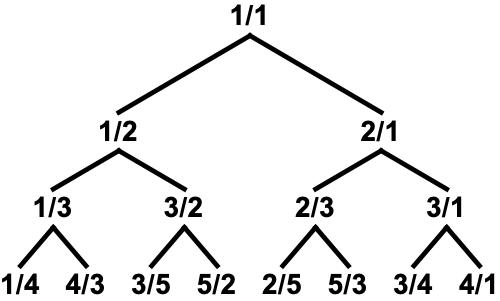}
\hspace{5mm}
\includegraphics[width=0.45\textwidth]{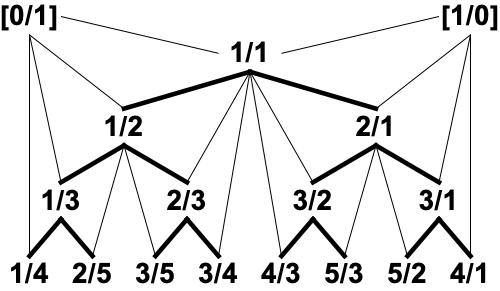}
\caption{The initial rows of the Calkin-Wilf tree (left) and Stern-Brocot tree (right).}
\label{fig:CWtreeSBtree}
\end{center}
\end{figure}

There are many exhaustive sequences of rationals other than (\ref{eq:farish}), one attractive 
option being the Calkin-Wilf tree \cite{CaWi}.  
The Calkin-Wilf tree is complete: it includes all positive rational numbers and each such number 
occurs precisely once.  The tree starts with the root value ${1}/{1}$,
and everything springs from this root (see Fig.~\ref{fig:CWtreeSBtree}, left panel).
Each rational in the tree has two ``children'': for the entry $m/n$,
the children are $m/(m+n)$ and $(m+n)/n$.  The ``left child''
$m/(m+n)$ is always smaller that $1$ while the ``right child''
$(m+n)/n$ is always greater that $1$
(mnemonic: the \emph{children} are ``top over sum'' and ``sum over bottom'').


The pattern of parity from one row of the Calkin-Wilf diagram to the next
is simple.  Denoting odd parity, even parity and no parity by $o$, $e$ and $n$
respectively, the parity transfer rules are as follows:
%
\begin{eqnarray}
\phantom{o}\phantom{\swarrow}        {\ e\ }\phantom{\searrow}\phantom{o}\qquad \quad 
\phantom{o}\phantom{\swarrow}        {\ o\ }\phantom{\searrow}\phantom{o}\qquad \quad 
\phantom{o}\phantom{\swarrow}        {\ n\ }\phantom{\searrow}\phantom{o}\qquad \quad 
\nonumber
\\
\phantom{o}        {\swarrow}\phantom{\ o\ }        {\searrow}\phantom{o}\qquad \quad
\phantom{o}        {\swarrow}\phantom{\ o\ }        {\searrow}\phantom{o}\qquad \quad
\phantom{o}        {\swarrow}\phantom{\ o\ }        {\searrow}\phantom{o}\qquad \quad
\label{eq:CWtransfer}
\\
        {e}\phantom{\swarrow}\phantom{\ o\ }\phantom{\searrow}        {o}\qquad \quad
          {n}\phantom{\swarrow}\phantom{\ o\ }\phantom{\searrow}        {e}\qquad \quad
            {o}\phantom{\swarrow}\phantom{\ o\ }\phantom{\searrow}        {n}\qquad \quad
\nonumber
\end{eqnarray}
%
We will now show that the elements of the Calkin-Wilf tree
are remarkably regular, with the pattern $(o, n, e)$ repeating interminably.
Thus, the parity of any specific term in the tree can easily be deduced.
We also prove that, with the ordering of $\QQ$ corresponding to the 
Calkin-Wilf tree, the three parity classes all have the same density.


\begin{theorem}
Let $\QQ^{+}$ be ordered with the Calkin-Wilf tree.
Then $\QQ^{+}$ may be partitioned into three parity classes,
$\QE^{+}$, $\QO^{+}$ and $\QN^{+}$, each having asymptotic density $\frac{1}{3}$.
\end{theorem}

\begin{proof}
The parity classes for the first few rows of the Calkin-Wilf tree are shown in
Fig.~\ref{fig:CWparitytree}. Odd rows follow a pattern $(one)^k o$ for some
$k \ge 0$ and even rows follow a pattern $ne(one)^k$. This is clearly true for the
first few rows. Using the transfer rules (\ref{eq:CWtransfer}), and arguing
inductively, it is clear that a row with pattern 
$(one)^k o$ is followed by a row with pattern $ne(one)^{2k}$. This, in turn,
is followed by a row with pattern $(one)^{4k+1} o$. The full sequence begins
$$
o\ \quad\ ne\ \quad\ (one) o\ \quad\ ne (one)^2\ \quad\
(one)^5 o\ \quad\ ne (one)^{10}\ \quad\ \cdots
$$
We conclude that, if the entire tree is written row by row as a sequence,
the parity follows an unvarying pattern, with \emph{odd} followed
by \emph{none} followed by \emph{even}.
The parity of an element at any position $N$ is immediately deduced
from $m = N (\mbox{mod}\ 3)$. The pattern also implies that the three
parity classes have equal densities.
\end{proof}

\begin{figure}[t]
\begin{center}
\includegraphics[width=0.8\textwidth]{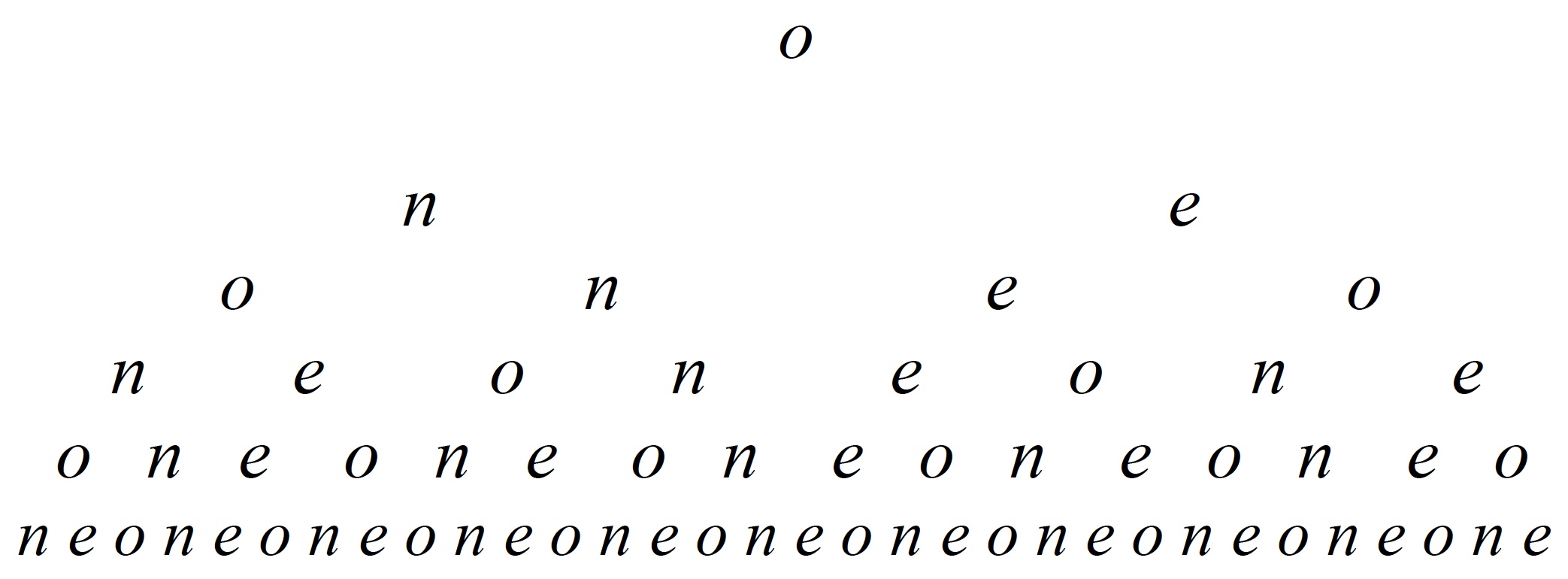}
\caption{Parity classes of terms in the Calkin-Wilf tree.}
\label{fig:CWparitytree}
\end{center}
\end{figure}

The Calkin-Wilf tree enumerates the positive rationals $\QQ^{+}$. This enumeration,
which we write $\{q_n:n\in\NN\}$, can be extended in a natural way to the full set of rationals:
we enumerate $\QQ$ by $\{0,q_1,-q_1,q_2,-q_2, \dots \}$. With this ordering,
the rational numbers split into three parts, each of asymptotic density $\frac{1}{3}$.

To summarise, the parity classes of elements of the Calkin-Wilf tree follow a simple
pattern if arranged in a single sequence: the pattern $\{o, n, e\}$ repeats indefinitely
(see Fig.~\ref{fig:CWparitytree}).
As a result, the densities of the parity classes in $\QQ$ are all equal for this ordering:
\begin{equation}
\rho_\QQ(\QE) = \rho_\QQ(\QO) = \rho_\QQ(\QN) = \tfrac{1}{3} \,.
\label{fig:Q-OneThird}
\end{equation}
%

The Stern-Brocot tree \cite[pg.~116]{GrKnPa} 
is another ordering of $\QQ$ very similar to the Calkin-Wilf tree.
The numbers at each level are formed from the mediants of adjacent pairs of numbers above
(Fig.~\ref{fig:CWtreeSBtree}, right panel). The mediant of two (reduced) rationals,
$m_1/n_1$ and $m_2/n_2$ is defined as $M(m_1/n_1, m_2/n_2):= (m_1+m_2)/(n_1+n_2)$.
We note that the parity of the mediants of two numbers of
different parity is the third parity:
\begin{equation}
M(e,o) = n \,,\qquad M(o, n) = e \,,\qquad M(n,e) = o \,.
\label{eq:Mtrans}
\end{equation}
We now show that, with the ordering of the Stern-Brocot tree, 
(\ref{fig:Q-OneThird}) holds true.

\begin{theorem}
For the order of $\QQ$ induced by the Stern-Brocot process,
the asymptotic density of each parity class, $\QE$, $\QO$ and $\QN$,
is $\frac{1}{3}$.
\end{theorem}

\begin{proof}
The Stern-Brocot tree is generated starting from level $0$ with the boundary elements
$\left[\frac{0}{1}\right]$ and $\left[\frac{1}{0}\right]$, representing $0$ and $\infty$
and with parities $[e]$ and $[n]$.
To get each subsequent level we add, between each pair of adjacent numbers,
the mediant of that pair, retaining all numbers already generated.
The results, for the first few levels, are shown in Fig.~\ref{fig:CWtreeSBtree}
(right panel).  The parity pattern for the first few levels is
$$
[e]o[n] \quad\ 
[e]noe[n] \quad\ 
[e]oneoneo[n] \quad\ 
[e]noenoenoenoenoe[n] \,.
$$
Using the transfer rules (\ref{eq:Mtrans}), an odd row with parity
$(eon)^{K}$ is followed by even row with parity $e(noe)^{2K-1}n$.
This in turn is followed by an odd row with parity $(eon)^{4K-1}$.
By an inductive argument, it follows that the parity pattern for an
odd row $k$ is $(eon)^{K}$,
where $K =   (2^k+1)/3$ and, for an even row $k$, is $e(noe)^{K}n$
where $K = (2^\ell-1)/3$.
This implies that the asymptotic densities are equal for all three parity classes.

The Stern-Brocot tree enumerates the positive rationals $\QQ^{+}$. This enumeration,
is easily extended to the full set of rationals, as was done above for the Calkin-Wilf tree.
Then the rational numbers split into three parts, each of asymptotic density $\frac{1}{3}$.
\end{proof}



The determination of the densities of parity classes for the ordering corresponding to
the Farey sequences is left as a challenge for readers.

\section{Conclusion.}
\label{sec:conclude}

We have extended parity from the integers to the rational
numbers. Three parity classes --- even, odd and `none' ---
were found.  The even and odd rationals $\QE$ and $\QO$ follow the
usual rules of parity.  The union of these forms an additive subgroup
$\QP$ of $\QQ$.

Using the 2-adic valuation, we partitioned $\QQ$ into subsets and
found a chain of subgroups, each having a quotient group of cosets.
We constructed a complete set of representatives for the cosets of $\QP$.

The Calkin-Wilf tree was found to have a remarkably simple parity
pattern, with the sequence `odd/none/even' repeating indefinitely.
Using the natural density, which provides a means of distinguishing
the sizes of countably infinite sets, we showed that, with the
Calkin-Wilf ordering, the three parity classes are equally dense
in the rationals. The same conclusion holds for the Stern-Brocot tree.

Finally, we remark that, while this study used only 2-adic numbers,
there is potential for broad extensions and generalizations using more general $p$-adic
numbers.


\subsection*{Acknowledgments.}
We are grateful to Tom Laffey, Emeritus Professor, School of Mathematics
\&\ Statistics, University College Dublin, for reading a draft of this paper
and to Tony O'Farrell, Emeritus Professor at Maynooth University
for guidance on $p$-adic numbers.




\end{document}